\documentclass[a4paper,12pt]{article}
\usepackage{latexsym}
\usepackage{amsmath}
\usepackage{amssymb}
\usepackage{hyperref}
\hypersetup{citecolor=red, linkcolor=blue, colorlinks=true}

\newtheorem{theorem}{Theorem}
\newtheorem{lemma}[theorem]{Lemma}

\newenvironment{proof}{{\sc Proof}:}{\hfill $\Box$} 

\def\cP{\mathcal{P}}
\def\Sym{\rm{Sym}}

\begin{document}

\title{On the proportion of elements of prime order in finite symmetric groups}
\author{Cheryl E.~Praeger\footnote{Centre for the Mathematics of Symmetry and Computation, University of Western Australia, 35 Stirling Highway, Perth 6009, Australia (\texttt{cheryl.praeger@uwa.edu.au}). The first author acknowledges funding from Australian Research Council Discovery Project grant DP190100450.}\ \  and Enoch Suleiman\footnote{Department of Mathematics, Federal University Gashua, Yobe State, Nigeria.
(\texttt{enochsuleiman@gmail.com})}}

\maketitle

\begin{abstract}

\begin{center}
{\em Dedicated to Daniela Nikolva\\ on the occasion of her $70^{th}$ birthday.}
\end{center}

\medskip
We give a short proof for an explicit upper bound on the proportion  of permutations of a given prime order $p$, acting on a set of given size $n$, which is sharp for certain $n$ and $p$. Namely, we prove that if $n\equiv k\pmod{p}$ with $0\leq k\leq p-1$, then this proportion is at most $(p\cdot k!)^{-1}$ with equality if and only if $p\leq n<2n$. \\

MSC2010:  Primary 20B30.\quad Secondary: 05A05\\

Keywords: finite symmetric groups, element proportions, elements of prime order
\end{abstract}

\section{Introduction}

The proportion of permutations of a given prime order $p$ and acting on a set of given size $n$ has been extensively studied. In particular, for fixed $p$ as $n$ grows, recursive formulas and an asymptotic expansion have been known for more than 70 years, first by Jacabsthal~\cite{J} in 1949, with extensions in the early 1950s by Chowla, Herstein and Scott~\cite{CHS} and Moser and Wyman~\cite{MW}. We give a brief discussion of these results in Section~\ref{sec:review}.  Our interest is in explicit estimates for this proportion in cases, even if they not asymptotically best possible. We obtain some simple upper bounds for this proportion which are sharp for certain $n, p$. Such explicit estimates are relevant in discussing various group theoretic algorithms, which we mention briefly in Section~\ref{sec:algs}. Our main result is the following. It is proved in Section~\ref{sec:proof}.

\begin{theorem}\label{thm:main}
Let $p$ be a prime and $n$ a positive integer, and write $n=ap+k$ where $a\geq0$ and $0\leq k<p$. Then the proportion  $\rho_p(n)$  of elements  of order $p$ in the symmetric group $S_n$ satisfies
\[
\rho_p(n)\leq \frac{1}{p\cdot k!} \quad \mbox{with equality if and only if $p\leq n<2p$.}
\]
\end{theorem}

\subsection{Asymptotic estimates}\label{sec:review}

Let $p$ be a prime and $n$ an integer such that $n\geq p$. Let $\rho_p^*(n)$ 
be the proportion of elements in the symmetric group $S_n$ of order dividing $p$, that is to say, elements $x\in S_n$ such that $x^p=1$. This proportion has been extensively studied in the literature, and asymptotic results are available which determine the value, when $p$ is fixed and $n$ is unbounded. The result of Jacobstahl  \cite{J} from 1949 gives the following expression for $\rho_p^*(n)$, and also a generating function for $\rho_p^*(n)$. 
\[
\rho_p^*(n) = \sum_{i=1}^{[n/p]}\frac{1}{(n-ip)!i!p^i},\quad \mbox{and}\quad \sum_{n=1}^\infty \rho_p^*(n) x^n = \exp(x)(\exp(x^p/p)-1).
\]
With such an expression for $\rho^*_p(n)$, it may not be immediately obvious what upper or lower bounds it implies. Indeed, while the expression for $\rho^*_p(n)$ could have been applied to obtain our result in Theorem~\ref{thm:main}, we believe that the short inductive proof we provide is helpful in this respect.

Several years after Jacobstahl's work appeared, Moser and Wyman~\cite[(3.41)]{MW} considered the proportion $\rho_p(n)$ of elements of order equal to $p$, and obtained an asymptotic expression:  
\[
\rho_p(n)\sim \frac{1}{\sqrt{p}\cdot n!}\left(\frac{n}{e}\right)^{n\left(1-1/p)\right)}e^{n^{1/p}}
 \]
which, in the light of Stirling's formula $n! \sim (2\pi n)^{1/2} \cdot (n/e)^n$, shows that for fixed $p$ as $n$ grows,  $\rho_p(n) \sim n^{-n/p}$. There are many similar asymptotic results, for example, where the primality condition on $p$ is removed, see in particular the paper of Wilf~\cite{W}. A helpful survey of these results is included in \cite[Section 2.2.6]{NPS13} or the explicit bounds in \cite{BGHP}. Most of these estimates are asymptotic, whereas it is sometimes preferable to have explicit upper and/or power bounds for such proportions.

\subsection{Some algorithmic considerations}\label{sec:algs}

For computational purposes a finite group $G$  may be given as a group of permutations or a group of matrices over a finite field, usually by specifying a small generating set $G=\langle X\rangle$.  We may know, or suspect, that $G$ is isomorphic to a group such as $S_n$, and we may wish to prove or exploit this computationally. 
Algorithms to recognise $G$, or to construct a standard generating set for $G$, typically seek special kinds of elements. These elements are usually sought, or constructed from, randomly selected elements, and to understand the complexity of such searches we need estimates for the proportions of various kinds of elements in the group. 

For example,  in \cite{BLNPS}, as part of constructing a standard generating set for a `black-box group' $G$ isomorphic to $S_n$, a  transposition was constructed by searching, via random selection, for an element $b\in G$ of order $2f$ for some odd positive integer $f\leq n^{18\log n}$, such that $b^f$ corresponds to a transposition in $S_n$. The choice of the upper bound $n^{18\log n}$ for $f$ was chosen so that the proportion of such elements was large enough   to find such an element $b$ with high probability (at least $0.318 n^{1/2}$, \cite[Theorem 5.1]{BLNPS}), and also to construct the transposition $b^f$ at a reasonable cost,   \cite[Proposition 6.1]{BLNPS}. The benefit of this method over a direct search for a transposition, by examining random elements, is obvious as the proportion of transpositions is only $(2\cdot (n-2)!)^{-1}$. The paper in \cite{BLNPS} contains an analogous analysis for construction of $3$-cycles in alternating groups. More details about algorithmic applications of proportions of elements in symmetric groups are given in \cite[Section 2.2.8]{NPS13}.

Let us turn now to primes larger than $2$ or $3$. Just as transpositions and $3$-cycles can be constructed by taking powers of elements with a single cycle of length $x$ and all other cycles of length coprime to $x$ (where $x=2$ or $3$) so also, for any prime $p$, we can construct a \emph{$p$-cycle} (that is, a permutation with one cycle of length $p$ and fixing all other points) by taking an appropriate power of a permutation $g$ which has a  single cycle of length $p$ and all other cycles of length coprime to $p$. Such elements $g$ are called \emph{pre-$p$-cycles}, and are useful for algorithms involving both permutation groups and classical matrix groups (see for example \cite{GPU, LNP09, NPP10}). Let $s_{\neg p}(n)$ denote the proportion of elements of $S_n$ with no cycle length divisible by $p$, equivalently the proportion of \emph{$p$-regular} elements - permutations with order coprime to $p$. It was proved in \cite[Lemma 3.1(ii)]{NPP10} that the proportion of pre-$p$-cycles in $S_n$ is $s_{\neg p}(n-p)/p$, and explicit upper and lower bounds for $s_{\neg p}(n-p)$ were obtained in \cite[Theorem 2.3{b}]{BLNPS}, with simpler bounds derived in \cite[Lemma 4.2]{LNP09} and \cite[Lemma 3.1(i)]{NPP10}. These bounds imply that the proportion of pre-$p$-cycles in $S_n$ lies between
\[
\frac{1}{4p\cdot (n-p)^{1/p}} \quad \mbox{and}\quad \frac{3}{p\cdot (n-p)^{1/p}}. 
\] 
This is far greater than the proportion of $p$-cycles, namely $(p\cdot (n-p)!)^{-1}$. 

Along the same lines,  we note that the proportion of \emph{$p$-singular} elements of $S_n$, that is to say, elements with order a multiple of $p$, is $1-s_{\neg p}(n)$ which, by our comments above, is greater than $1-3 n^{-1/p}$. The simple upper bound given in Theorem~\ref{thm:main} shows that the proportion of $p$-singular elements of $S_n$ is far greater than $\rho_p(n)$.
This is an easy confirmation that constructing elements of order $p$ by taking powers of $p$-singular elements is much more efficient than searching for such elements directly by random selection.

\section{Proof of the main theorem}\label{sec:proof}
Let $n$ be a positive integer, and let $\Omega=\{1,\dots,n\}$ and $\Sym(\Omega)$ be the symmetric group $S_n$ on $\Omega$.   Let $\cP(n,p)$ denote the subset of  $\Sym(\Omega)$ consisting of all the elements of order $p$, so that 
\begin{equation}\label{e:rho}
\rho_p(n):=\frac{|\cP(n,p)|}{n!}.
\end{equation}
Note that here we define this set, and proportion, for all positive integers $n$; if $n<p$ then of course $\cP(n,p)$ is empty and $\rho_p(n)=0$. First we record this and another basic fact.

\begin{lemma}\label{lem1}
Let $n, p, \Omega, \cP(n,p)$ and $\rho_p(n)$ be as above.
\begin{enumerate}
\item[(a)] If $n<p$ then $\rho_p(n)=0$, and $\rho_p(p)=1/p$.

\item[(b)] If $n\geq p$, then, for each subset $\Delta\subseteq\Omega$ with $|\Delta|=p$, there are exactly $(p-1)!$ pairwise distinct $p$-cycles permuting the points of $\Delta$ (that is, $p$-cycles in $\Sym(\Delta)$).
\end{enumerate}
\end{lemma}

\begin{proof}
Part (a)  for $n<p$ follows from the discussion before the statement. So assume that $n\geq p$ and let $\Delta=\{\delta_1,\dots,\delta_p\}$ be a $p$-element subset of $\Omega$. Each $p$-cycle in $\Sym(\Delta)$ permutes all of the points of $\Delta$ in a single cycle and has a unique expression of the form $(\delta_1,\alpha_2,\dots,\alpha_p)$ where $\alpha_2,\dots,\alpha_p$ are precisely the points in $\Delta\setminus\{\delta_1\}=\{\delta_2,\dots,\delta_p\}$ in some order. To count the number of possibilities for the $p$-cycle, we observe that there are exactly $p-1$ choices for $\alpha_2\in \Delta\setminus\{\delta_1\}$ and, given $\alpha_2$, there are exactly $p-2$ choices for $\alpha_3$ from $\Delta\setminus\{\delta_1,\alpha_2\}$, and so on. We see that there are therefore exactly $(p-1)\times (p-2)\times\dots\times 1 = (p-1)!$ of these $p$-cycles in $\Sym(\Delta)$. 

Finally, if $n=p$, then the elements of $\Sym(\Omega)$ of order $p$ are precisely the $p$-cycles on $\Omega$, and we have just shown that there are exactly $(p-1)!$ such elements. Therefore $|\cP(p,p)|=(p-1)!$ and $\rho_p(p)
= |\cP(p,p)|/p! = 1/p$.
\end{proof}

\medskip
We note that Lemma~\ref{lem1} proves Theorem~\ref{thm:main} for the smallest values of $n$, namely for $1\leq n\leq p$. The main proof will be by induction on $n$ with these  as the `base cases'. It will use the following recursion for $\rho_p(n)$.

\begin{lemma}\label{lem2}
Let $p$ be a prime and $n$ an integer such that $n\geq p+1$. Then
\[
n\cdot \rho_p(n) = \rho_p(n-1) + \rho_p(n-p) + \frac{1}{(n-p)!}.
\]
\end{lemma} 

\begin{proof}
By \eqref{e:rho}, $|\cP(n,p)|=n!\rho_p(n)$ for each positive integer $n$. 
We establish the recursion by enumerating $\cP(n,p)$ as follows.  We partition $\cP(n,p)$ as $\cP_1(n,p)\cup\cP_{2}(n,p)$, where $\cP_1(n,p)$ consists of all elements $g\in\cP(n,p)$ such that $1^g=1$, and $\cP_{2}(n,p)$ consists of all elements $g\in\cP(n,p)$ such that $1^g\ne1$. Now $\cP_1(n,p)$ is precisely the set of elements of order $p$ in $\Sym(\Delta)$, where $\Delta=\{2,3,\dots,n\}$ and hence $|\cP_1(n,p)| = (n-1)! \rho_p(n-1)$, by \eqref{e:rho}.

Consider now the complement  $\cP_{2}(n,p)$. To enumerate the elements of $\cP_{2}(n,p)$, we note that, for each such element $g$, the point $1$ lies in a cycle $h$ of $g$ of length $p$, since $1^g\ne 1$. The number of such cycles is equal to the number $\binom{n-1}{p-1}$ of subsets $\Delta'$ of $(p-1)$-element subsets of $\Omega\setminus\{1\}$ such that the $p$-cycle $h$ permutes the points of $\Delta:=\Delta'\cup\{1\}$, times the number $(p-1)!$ of $p$-cycles in $\Sym(\Delta)$ (using Lemma~\ref{lem1}(b)).  Then, for each of these 
    $\binom{n-1}{p-1}(p-1)!$ cycles, say $h$ on a subset $\Delta$  containing the point $1$, the elements $g\in \cP_{2}(n,p)$ which have $h$ as a $p$-cycle are precisely the permutations $g=hg'$ where $g'$ is an element of $\Sym(\Omega\setminus\Delta)=S_{n-p}$ of order dividing $p$. The number of such elements $g'$ is equal to the number $|\cP(n-p,p)|=(n-p)! \rho_p(n-p)$ of elements of $S_{n-p}$ of order $p$, together with the identity element.  Thus
    \begin{align*}
    |\cP_{2}(n,p)| &= \binom{n-1}{p-1}(p-1)! \left( (n-p)!\, \rho_p(n-p) + 1 \right)\\ 
    &= (n-1)! \left(  \rho_p(n-p) + \frac{1}{(n-p)!} \right).
    \end{align*}
    It now follows that 
    \begin{align*}
 n!\rho_p(n) &=  |\cP(n,p)| =  |\cP_1(n,p)| + |\cP_{2}(n,p)| \\
 & = (n-1)! \rho_p(n-1) + (n-1)! \left(  \rho_p(n-p) + \frac{1}{(n-p)!} \right)
    \end{align*}  
and hence $n\cdot \rho_p(n) =   \rho_p(n-1) + \rho_p(n-p) + \frac{1}{(n-p)!}$, completing the proof.  
\end{proof}

\medskip
We now use Lemma~\ref{lem1} and~\ref{lem2} to prove Theorem~\ref{thm:main}.

\medskip
{\sc Proof of Theorem} \ref{thm:main}.\quad Let $n$ be a positive and let $a, k$ be the (unique) integers such that $n=ap+k$ where $a\geq 0$ and $0\leq k\leq p-1$. We will prove that 
\[
\rho_p(n)\leq \frac{1}{p\cdot k!} \quad \mbox{with equality if and only if $p\leq n<2p$.}
\]
This assertion follows from Lemma~\ref{lem1}(a) if $n\leq p$, so assume that $n\geq p+1$, and assume inductively that the result holds for all integers strictly less than $n$. Then, by Lemma~\ref{lem2}, 
\[
\rho_p(n)= \frac{\rho_p(n-1)}{n} + \frac{\rho_p(n-p)}{n} +\frac{1}{n\cdot (n-p)!}
\]
Now our standard representations for $n-p$ and $n-1$ are $n-p=(a-1)p+k$, and $n-1 = ap+(k-1)$ if $1\leq k\leq p-1$ and $n-1=(a-1)p + (p-1)$ if $k=0$. 

Suppose first that $p+1\leq n\leq 2p-1$, that is to say, $a=1$ and $1\leq k\leq p-1$. Then $n-p=k<p$ so $\rho_p(n-p)=0$ by Lemma~\ref{lem1}(a), and by induction, $\rho_p(n-1) = 1/(p\cdot (k-1)!)$. Thus 
\[
\rho_p(n)= \frac{1}{np\cdot (k-1)!} + 0 +\frac{1}{n\cdot k!} = \frac{k+p}{np\cdot k!} = \frac{1}{p\cdot k!}
\]
and the result holds in these cases. 
Thus we may assume that $n\geq 2p$. If $k>0$ then, by induction, $\rho_p(n-1) \leq 1/(p\cdot (k-1)!)$  and $\rho_p(n-p)\leq 1/(p\cdot k!)$, so that  
\[
\rho_p(n)\leq  \frac{1}{np\cdot (k-1)!} + \frac{1}{np\cdot k!} +\frac{1}{n\cdot (n-p)!} = \frac{k+1+ p\cdot k!/(n-p)!}{np\cdot k!}.
\]
Since $n-p\geq p\geq k+1$, we have $p\cdot k!/(n-p)!\leq 1$, and so the numerator $k+1+ p\cdot k!/(n-p)!\leq k+2\leq p+1 < n$, so $\rho_p(n) < 1/(p\cdot k!)$, and the result follows in these cases. Suppose finally that $n\geq 2p$ and $k=0$, that is, $n=ap$ with $a\geq2$. Then, by induction, $\rho_p(n-p)\leq 1/p$ and $\rho_p(n-1) \leq  1/(p\cdot (p-1)!) = 1/p!$. Thus
\[
\rho_p(n)\leq  \frac{1}{n\cdot p!} + \frac{1}{np} +\frac{1}{n\cdot (n-p)!} = \frac{1}{np}\left( \frac{1}{(p-1)!} + 1 +\frac{p}{(n-p)!}\right).
\]
 Since $n-p\geq p$ we have $p/(n-p)! \leq 1/(p-1)!$ and so 
\[ 
 \rho_p(n) \leq \frac{1}{np}\left( \frac{2}{(p-1)!} +1\right) \leq \frac{3}{np} < \frac{1}{p}
 \]
and this completes the proof by induction of Theorem~\ref{thm:main}.

\end{document}